\documentclass[12pt]{amsart}
\usepackage{amssymb, hyperref, mathrsfs,xcolor, cancel, comment}

\usepackage[margin=1.1in]{geometry}

\newtheorem{theoremIntro}{Theorem}[]

\newtheorem{questionIntro}{Question}

\newtheorem{theorem}{Theorem}[section]
\newtheorem{lemma}[theorem]{Lemma}
\newtheorem{proposition}[theorem]{Proposition}

\newtheorem{definition}[theorem]{Definition}
\newtheorem{corollary}[theorem]{Corollary}

\theoremstyle{remark}
\newtheorem{remark}[theorem]{Remark}

\newcommand{\calG}{\ensuremath{\mathcal{G}}}
\newcommand{\calS}{\ensuremath{\mathcal{S}}}

\newcommand{\calA}{\ensuremath{\mathcal{A}}}
\newcommand{\calB}{\ensuremath{\mathcal{B}}}

\newcommand{\xra}{\xrightarrow}

\newcommand{\gln}{\ensuremath{\operatorname{GL}}}

\newcommand{\mo}{{-1}}

\newcommand{\bbZ}{\ensuremath{\mathbb{Z}}}
\newcommand{\bbC}{\ensuremath{\mathbb{C}}}

\newcommand{\bbI}{\ensuremath{\mathbb{I}}}
\newcommand{\bbN}{\ensuremath{\mathbb{N}}}

\newcommand{\cyc}{\ensuremath{\mathrm{cyc}}}
\newcommand{\ab}{\ensuremath{\mathrm{ab}}}
\begin{document}

\title{On additive co-minimal pairs}
 
\author{Arindam Biswas}
\address{Universit\"at Wien, Fakult\"at f\"ur Mathematik, Oskar-Morgenstern-Platz 1, 1090 Wien, Austria}
\curraddr{Department of Mathematics, Technion - Israel Institute of Technology,
Haifa 32000,
Israel}
\email{biswas@campus.technion.ac.il}
\thanks{}

\author{Jyoti Prakash Saha}
\address{Department of Mathematics, Indian Institute of Science Education and Research Bhopal, Bhopal Bypass Road, Bhauri, Bhopal 462066, Madhya Pradesh,
India}
\curraddr{}
\email{jpsaha@iiserb.ac.in}
\thanks{}



\begin{abstract}
A pair of non-empty subsets $(W,W')$ in an abelian group $G$ is an additive complement pair if $W+W'=G$. $W'$ is said to be minimal to $W$ if $W+(W'\setminus \{w'\}) \neq G, \forall \,w'\in W'$. In general, given an arbitrary subset in a group, the existence of minimal complement(s) depends on its structure. The dual problem asks that given such a set, if it is a minimal complement to some subset. Additive complements have been studied in the context of representations of integers since the time of Erd\H{o}s, Hanani, Lorentz and others. The notion of minimal complements is due to Nathanson. We study tightness property of complement pairs $(W,W')$ such that both $W$ and $W'$ are minimal to each other. These are termed co-minimal pairs and we show that any non-empty finite set in an arbitrary free abelian group belongs to some co-minimal pair. We also study infinite sets forming co-minimal pairs. At the other extreme, motivated by unbounded arithmetic progressions in the integers, we look at sets which can never be a part of any minimal pair. This leads to a discussion on co-minimality, subgroups, approximate subgroups and asymptotic approximate subgroups of $G$.
\end{abstract}

\maketitle

\section{Introduction}

Let $(G, \cdot)$ be a group and let $A,B$ be non-empty subsets of $G$ with $A\cdot B = G$. Then the set $A$ is said to be a \textit{left complement} of $B$ in $G$ (respectively, $B$ is a \textit{right complement} of $A$ in $G$) and the pair $(A,B)$ is said to be a \textit{complement pair} in $G$. A left (resp. right) complement $A$ of some non-empty subset $B$ of $G$ is said to be \textit{minimal} if $A\cdot B=G$ (respectively $B\cdot A=G$) and $(A\setminus\lbrace a\rbrace)\cdot B\neq G$ (respectively $B\cdot (A\setminus \lbrace a \rbrace)\neq G$) for each $a\in A$. A complement pair in which at least one subset is minimal will be called a \textit{minimal pair}.

In the case of abelian groups, a left complement of a subset is also a right complement to that subset and vice versa and these are also known as additive complements.  Also, any non-empty subset $A$ is always a part of some complement pair (for instance, consider the pair $(A, G)$). 

Additive complements have been studied since a long time in the context of representations of the integers e.g., they appear in the works of Erd\H{o}s, Hanani, Lorentz and others. See \cite{Lorentz54}, \cite{Erdos54}, \cite{ErdosSomeUnsolved57} etc. A famous conjecture of Erd\H{o}s--Turan \cite{ET41} on additive bases states that given any (asymptotic) additive base $\beta$ of the natural numbers, of order $h$\footnote{A subset $\beta$ is called an (asymptotic) additive basis of order $h$ if there is some positive integer $h$ such that every sufficiently large positive integer $n$ can be written as the sum of at most $h$ elements of $\beta$.}, the number of representations of a positive integer $n$, as a function of $n$, must tend to $\infty$. Another important direction of research in this area is the study of sum-free sets, e.g., \cite{GR05}, \cite{Tra18} etc. The notion of minimal additive complements is due to Nathanson, who introduced it in \cite{Nat11} in the course of his study of natural arithmetic analogues of the metric concept of nets in the setting of the integers and also groups in general.  Henceforth, by a complement we shall mean an additive complement. If we need to use a set-theoretic complement, we shall explicitly state it.

The situation becomes interesting when we ask whether a given subset admits a minimal complement or not (posed by Nathanson in \cite{Nat11}), and also the dual question whether a given subset could be a minimal complement to some subset in a group or not. In this article, our aim is to study these questions. 

\subsection{Statement of results} We show that,
\begin{theoremIntro}
\label{Thm:FiniteSubsetIsAMinComp}
Any non-empty finite subset of a free abelian group (not necessarily of finite rank) is a minimal complement to some subset. 
\end{theoremIntro}

In fact, one can go beyond and study tightness property of a set and its minimal complement. For this, we introduce the notion of co-minimal pairs. 

\begin{definition}[Co-minimal pair]\label{DefnCoMinimal}
Let $A,B\subseteq G$ be two non-empty subsets. Then the pair $(A,B)$ is defined to be a \textnormal{co-minimal pair} if $A$ is a left minimal complement of $B$ and $B$ is a right minimal complement of $A$. 
\end{definition}

Thus, a co-minimal pair is a complement pair in which both the subsets are minimal. The existence of co-minimal pairs is a stronger notion to the existence of a minimal complement. In this context, we show that 
\begin{theoremIntro}
\label{Thm:CoMinimal}
Any non-empty finite subset of an arbitrary free abelian group is a part of some co-minimal pair. Moreover, if $S$ is a two element subset of a group $G$, then $(S, R), (L, S)$ are co-minimal pairs for some subsets $L, R$ of $G$.
\end{theoremIntro}

It should be emphasised that the dual question of whether a given subset is a minimal complement to some subset is often harder to answer than the existence of minimal complements. For instance, in the case of a finite group $G$, it is easy to see that any subset $A$ has a minimal complement. However, it is not clear (and in fact false in general), whether it is a minimal complement to some subset. This is made precise in the following theorem which states that the large subsets of a group cannot be a minimal complement to some subset and in particular cannot belong to a co-minimal pair.

\begin{theoremIntro}
\label{Thm:NecessaryCond}
Let $G$ be a group containing at least $3k+1$ elements for some integer $k\geq 1$. Then no subset of $G$ of size $|G|-k$ (i.e. having exactly $k$ elements in its set theoretic complement in $G$) can be a minimal complement to some subset in $G$. Consequently, if $X$ is a proper subset of a finite group $G$ such that $X$ is a part of a co-minimal pair, then 
$$
|X| 
< 
\frac 23 |G| + \frac 13
$$
holds. 
\end{theoremIntro}

Thus, for large, finite subsets, the property of being a minimal complement does not hold inside finite groups in contrast with the situation in the free abelian groups (and in particular $\mathbb{Z}$) cf. Theorem \ref{Thm:FiniteSubsetIsAMinComp}. The above Theorem \ref{Thm:NecessaryCond} also has consequences in the case of infinite groups. See Corollary \ref{Cor:NecessaryCond}.

In section \ref{sec3}, we study several properties of co-minimal pairs. Next, in section \ref{sec4} we turn our attention to infinite subsets $A,B\subseteq G$ forming co-minimal pairs. For this, the notion of spiked subsets (see Definition \ref{Spiked subsets} and also section \ref{sec4}) is useful. If $G_1, G_2$ are subgroups of an abelian group $G$ such that the multiplication map $G_1 \times G_2 \to G$ defined by $(g_1, g_2)\mapsto g_1g_2$ is an isomorphism, then it turns out that the subsets of $G$ of the form $B\times G_2$ with $B\subseteq G_1$ is a part of a co-minimal pair in $G$ if and only if $B$ is a part of co-minimal pair in $G_1$ (see Lemma \ref{Lemma:Spike}). More generally, an appropriate analogue of this statement also holds for spiked subsets.

\begin{theoremIntro}[Theorem \ref{Thm:Spiked}]
Let $G_1, G_2$ denote two subgroups of an abelian group $\calG$. Let $X$ be a $(u, \varphi)$-bounded spiked subset of $\calG$ with respect to $G_1, G_2$ and with base $\calB$. If $u$ admits a $\varphi$-moderation, then $X$ is a part of a co-minimal pair in $G_1G_2$ of the form $(X, M_v)$ where $M_v$ is the graph of the restriction of a moderation $v$ of $u$ to some subset $M$ of $G_1$ if and only if $X$ is equal to $\calB G_2$ and $\calB$ is a part of a co-minimal pair in $G_1$. 
\end{theoremIntro}
 Roughly speaking, the above theorem classifies all the spiked subsets which can be a part of a co-minimal pair of certain form.

On the other hand, if we take $A=\mathbb{N}^{d}$ in $G=\mathbb{Z}^{d}$, then $A$ can never have a minimal complement and is also not a minimal complement to any set in $G$. Thus it is in a sense the other extreme to being a part of a co-minimal pair. This observation can be generalised to the following result.

\begin{theoremIntro}
No unbounded, generalised arithmetic progressions, in free abelian groups can belong to a minimal pair.
\end{theoremIntro} 

This is detailed in Theorem \ref{semilinear}. The above leads to a discussion on co-minimality and infinite approximate subgroups and asymptotically approximate subgroups. See section \ref{sec5}.
 
Finally, in the section on concluding remarks and further questions (see section \ref{sec6}), we look at self-complements\footnote{A non-empty set $A$ is a self-complement iff $A\cdot A = G$. If, in addition, $A$ is minimal, then it's called a minimal self-complement or a co-minimal pair $(A,A)$.} in arbitrary abelian groups and remark that a set $A$ is a minimal self-complement iff $A$ does not contain any non-trivial $3$ term arithmetic progression. We finish by stating several open questions and further directions of research. 

\section{Finite subsets and co-minimal pairs}\label{sec2}

We begin the section by noting that the set of all co-minimal pairs is a strict subset of the set of all minimal pairs.

\begin{lemma}
There exists a non-empty subset $A\subseteq \mathbb{Z}$ such that $A$ has a minimal complement, but $A$ is not a minimal complement to any set. Thus, $A$ can belong to a minimal pair, but can never belong to a co-minimal pair.
\end{lemma}

\begin{proof}
	Consider the subset $A=2\mathbb{Z}\cup\lbrace 1\rbrace\subseteq \mathbb{Z}$. Then $$\inf A = -\infty \text{ and } \sup A = +\infty.$$ By a result of Chen--Yang \cite[Theorem 1]{CY12}, the set $A$ admits a minimal complement in $\bbZ$. However, $A$ itself cannot be a part of a co-minimal pair. For this, we show a stronger statement that $A$ cannot be a minimal complement to any set in $\mathbb{Z}$. Otherwise, suppose $A$ is a minimal complement to $B$ for some subset $B$ of $\bbZ$. If all the elements of $B$ have the same parity, then replacing $B$ by $B+1$ if necessary, we may assume that $B$ is a subset of $2\bbZ$ (see Lemma \ref{Lemma:CoMinimalTranslate} why this doesn't change the existence of minimal complements). Since $A+B$ is equal to $\bbZ$, it follows that $B$ is equal to $2\bbZ$. However, $2\bbZ\cup \{1\}$ is not a minimal complement to $2\bbZ$ since the subset $\{0,1\}$ of $2\bbZ\cup \{1\}$ is a complement to $2\bbZ$. Let us assume that $B$ contains two elements of different parity. Let $B_e$ (resp. $B_o$) denote the subset of $B$ consisting of the even (resp. odd) elements of $B$. Note that $2\bbZ + B_e$ and $2\bbZ + B_o$ are subsets of  $A+ B$ and 
	$$2\bbZ \subseteq  2\bbZ + B_e \text{ and } 1 + 2\bbZ \subseteq 2\bbZ + B_o.$$ Thus, $$\mathbb{Z} = 2\mathbb{Z}\cup ( 1 + 2\bbZ ) \subseteq  ( 2\bbZ + B_e ) \cup ( 2\bbZ + B_o ) = 2\mathbb{Z} + (B_e \cup B_o) .$$
      This implies that the subset $2\bbZ$ of $A$ is a complement to $B_e \cup B_o = B$. Hence $A$ is not a minimal complement to $B$. Consequently, $A$ cannot be a part of a co-minimal pair. 
\end{proof}
Further, co-minimal pairs are preserved under translations.

\begin{lemma}
\label{Lemma:CoMinimalTranslate}
If $(A, B)$ is a co-minimal pair in a group $G$, then $(g\cdot A, B\cdot h)$ forms a co-minimal pair for any two elements $g, h$ of $G$.
\end{lemma}

\begin{proof}
It follows since multiplication by an element $g$ from the left induces a bijection from $G$ to $G$ and multiplication by an element $h$ from the right also induces a bijection.
\end{proof}

Next, we will proceed to prove Theorems \ref{Thm:FiniteSubsetIsAMinComp}, \ref{Thm:CoMinimal}. 

\begin{proposition}
\label{Prop:ImageLarge}
Given a non-empty finite subset $S$ of $\bbZ^n$ for $n\geq 2$, there exists an automorphism $\varphi$ of the group $\bbZ^n$ such that the image of the set $\varphi(S)$ under the projection map $$\pi_1: \bbZ^n\to \bbZ \text{ (onto the first coordinate) }$$ contains exactly $\# S$ elements. 
\end{proposition}

\begin{proof} We show it by induction. The base case is for $n=2$.
Suppose $S$ is a non-empty finite subset of $\bbZ^2$. Then for some positive integer $m$, the image of the set $\varphi_m(S)$ under the projection map
 $$\pi_1: \bbZ^2\to \bbZ \text{ (onto the first coordinate) }$$ contains $\# S$ elements where $\varphi_m$ denotes the automorphism of $\bbZ^2$ defined by 
$$\varphi_m:= 
\begin{pmatrix}
1 & m \\
0 & 1
\end{pmatrix}.
$$
Otherwise, there exist infinitely many positive integers $m_1 < m_2 < m_3 < \cdots$ and two distinct elements $s, t$ in $S$ such that 
\begin{equation}
\label{Eqn:pi1Eq}
\pi_1(\varphi_{m_i}(s)) = \pi_1(\varphi_{m_i}(t)) \text{ for any } i\geq 1.
\end{equation}
Let $s$ (resp. $t$) be equal to $(a_1, a_2)$ (resp. $(b_1, b_2)$). So any integer $i\geq 1$, we obtain 
$$a_1 + m_i a_2 = b_1 + m_i b_2.$$
Note that $a_2 \neq b_2$ (otherwise $a_1 = b_1$ and hence $s = t$). Thus the equality in Equation \eqref{Eqn:pi1Eq} cannot hold for infinitely many positive integers $m_i$. So the image of the set $\varphi_m(S)$ under the projection map 
$$\pi_1: \bbZ^2\to \bbZ \text{ (onto the first coordinate) }$$ 
contains $\# S$ elements for some positive integer $m$, i.e., the Proposition holds for $n=2$. 
 
Suppose the Proposition also holds for $n=r$ for some positive integer $r\geq 2$. Let $S$ be a non-empty finite subset of $\bbZ^{r+1}$. Then for some positive integer $m$, the image of the set $\varphi_m(S)$ under the projection map $$\pi: \bbZ^{r+1}\to \bbZ^r \text{ (onto the first $r$-coordinates) }$$ contains $\# S$ elements where $\varphi_m$ denotes the automorphism of $\bbZ^{r+1}$ defined by 
$$\varphi_m:= 
\begin{pmatrix}
1 & 0 & \cdots & 0 & m \\
0 & 1 & \cdots & 0 & m \\
\vdots & \vdots & \ddots & \vdots & \vdots \\
0 & 0 & \cdots & 1 & m\\
0 & 0 & \cdots & 0 & 1
\end{pmatrix}_{(r+1)\times (r+1)}.
$$
Otherwise, there exist infinitely many positive integers $m_1 < m_2 < m_3 < \cdots$ and two distinct elements $s, t$ in $S$ such that 
\begin{equation}
\label{Eqn:piEq}
\pi(\varphi_{m_i}(s)) = \pi(\varphi_{m_i}(t)) \text{ for any } i\geq 1.
\end{equation}
Let $s$ (resp. $t$) be equal to $(a_1, a_2, \cdots, a_{r+1})$ (resp. $(b_1, \cdots, b_{r+1})$). So any integer $i\geq 1$ and $1\leq \ell \leq r$, we obtain 
$$a_\ell + m_i a_{r+1} = b_\ell + m_i b_{r+1}.$$
Note that $a_{r+1} \neq b_{r+1}$ (otherwise $a_\ell = b_\ell$ for $1\leq \ell\leq r$ and hence $s = t$). Thus the above equality in Equation \eqref{Eqn:piEq} cannot hold for infinitely many positive integers $m_i$. So the image of the set $\varphi_m(S)$ under the projection map $$\pi: \bbZ^{r+1}\to \bbZ^r \text{ (onto the first $r$-coordinates) }$$ contains $\# S$ elements for some positive integer $m$.
By the induction hypothesis, there exists an element $A\in \gln_r(\bbZ)$ such that the image of $\pi(\varphi_m(S))$ under $A$ contains $\# \pi(\varphi_m(S)) = \# S$ elements. \\
Let $\widetilde A$ denote the automorphism of $\bbZ^{r+1} = \bbZ^r \times \bbZ$ which acts by $A$ on the first factor $\bbZ^r$ and acts trivially on the second factor $\bbZ$. Then the image of $(\widetilde A\circ \varphi_m)(S)$ under the projection map $$\pi: \bbZ^{r+1} \to \bbZ \text{ (onto the first coordinate) } $$contains exactly $\# S$ elements. Hence the Proposition follows. 
\end{proof}

\begin{lemma}
\label{Lemma:MinCompProj}
Let $S$ be a nonempty finite subset of an abelian group $G$. Suppose $G_1$ is a group, and $\pi:G\to G_1$ is a surjective group homomorphism such that $\pi(S)$ contains $\# S$ elements. Then $S$ is a minimal complement to some subset of $G$ if the subset $\pi(S)$ of $G_1$ is a minimal complement to some subset $W$ of $G_1$. 
\end{lemma}

\begin{proof}
Since $G$ is abelian, it follows that $S$ is a complement to $\pi^\mo(W)$. Moreover, since $S$ is a minimal complement to $W$ and the image of $S$ under $\pi:G\to G_1$ contains $\# S$ elements, the set $S$ is a minimal complement to $\pi^\mo(W)$. 
\end{proof}

\begin{proof}[\textbf{Proof of Theorem \ref{Thm:FiniteSubsetIsAMinComp}}]
Let $S$ be a nonempty finite subset of a free abelian group $G$. If $G$ has finite rank, then we identify $G$ with $\bbZ^n$ with $n= \mathrm{rk} G$. When $n$ is equal to one, the result follows from \cite[Theorem 9]{Kwon19}. Moreover, when $n\geq 2$, by Proposition \ref{Prop:ImageLarge}, there exists an automorphism $\varphi$ of $G = \bbZ^n$ such that the image of $\varphi(S)$ under the the projection map 
$$ \pi:\bbZ^n\to \bbZ \text{ (onto the first coordinate) } $$
contains exactly $\# S$ elements. Note that $\pi(\varphi(S))$ contains $\# \varphi(S)$ elements and by \cite[Theorem 9]{Kwon19}, the subset $\pi(\varphi(S))$ of $\bbZ$ is a minimal complement to some subset of $\bbZ$. Hence by Lemma \ref{Lemma:MinCompProj}, $\varphi(S)$ is a minimal complement to some subset $W$ of $\bbZ^n$, and hence $S$ is a minimal complement to $\varphi^\mo(W)$. Consequently, any non-empty subset of any free abelian group of finite rank is a minimal complement to some subset. 

Moreover, when $G$ has infinite rank, i.e., when $G$ is isomorphic to the direct product $\bbZ^I$ for an infinite set $I$, it follows that for some finite subset $J$ of $I$, the image of $S$ under the projection map 
$\pi: \bbZ^I \to \bbZ^J$ (obtained by restricting the elements of $\bbZ^I$ (considered as the group of all functions from $I$ to $\bbZ$) to $J$) contains exactly $\# S$ elements. Indeed, if for each pair $(s,t)$ with $s, t\in S$, choose an element $i_{s,t} \in I$ such that $s, t$ take different values at $i_{s,t}$, then $J$ can be taken to be $$J = \{i_{s,t} \,|\, (s, t) \in S\times S, s\neq t\}.$$ 
By the conclusion of the previous paragraph, it follows that $\pi(S)$ is a minimal complement to some subset of $\bbZ^J$. Hence by Lemma \ref{Lemma:MinCompProj}, it follows that $S$ is a minimal complement to some subset of $G$. This completes the proof of Theorem \ref{Thm:FiniteSubsetIsAMinComp}.
\end{proof}

Now, we turn our attention to co-minimal pairs (cf. Definition \ref{DefnCoMinimal}) which measures the tightness property of a set and its complement. To show Theorem \ref{Thm:CoMinimal}, we need a result from a prior work.
\begin{theorem}[{\cite[Theorem 2.1]{MinComp1}}]\label{finMin}
Let $G$ be an arbitrary group with $S$ a nonempty finite subset of $G$. Then every complement of $S$ in $G$ contains a minimal complement to $S$. 
\end{theorem}

This implies that a non-empty finite set $S\subseteq G$ belongs to some minimal pair. Using the above Theorem \ref{finMin} and Theorem \ref{Thm:FiniteSubsetIsAMinComp} we shall show that it also belongs to a co-minimal pair when $G$ is a free abelian group. This will establish Theorem \ref{Thm:CoMinimal}.

\begin{proof}[\textbf{Proof of Theorem \ref{Thm:CoMinimal}}]
Let $A$ be a non-empty finite subset of a free abelian group $G$. By Theorem \ref{Thm:FiniteSubsetIsAMinComp}, $A$ is a minimal complement to some subset $B$ of $G$. By Theorem \ref{finMin}, any complement $C$ of a non-empty finite subset $W$ of a group contains a minimal complement to $W$. Consequently, $B$ contains a minimal complement $B'$ of $A$. Since $A+B'=G$, $A$ is a minimal complement to $B$ and $B$ contains $A$, it follows that $A$ is a minimal complement to $B'$ . Hence $(A, B')$ is a co-minimal pair. 

Let $S = \{g, h\}$ be a two-element subset of a group $G$. Then $S$ is a minimal complement to $G\setminus \{g^\mo h\}$. Indeed, 
$$g\cdot (G\setminus \{g^\mo h\})= G \setminus \{h\}$$ 
and 
$$h \in G \setminus \{hg^\mo h\} 
= (G\cdot h) \setminus \{hg^\mo h\} 
= 
h\cdot (G\setminus \{g^\mo h\}),$$ 
which implies that $S$ is a minimal left complement to $G\setminus \{g^\mo h\}$. By \cite[Theorem 2.1]{MinComp1}, $G\setminus \{g^\mo h\}$ contains a minimal right complement to $S$. Hence $(S, R)$ is a co-minimal pair for some subset $R$ of $G$. 

Note that the proof of \cite[Theorem 2.1]{MinComp1} can be suitably adapted to prove that for a non-empty finite subset $S$ of a group $G$, every left complement to $S$ contains a minimal left complement to $S$. Using this result and an argument similar to the above, it follows that $(L, S)$ is a co-minimal pair for some subset $L$ of $G$. 
\end{proof}

\begin{remark}
Let $G$ be a finitely generated abelian group isomorphic to the direct product of its torsion part $G_{\mathrm{tors}}$ and a free abelian group $F$. If $A$ is a non-empty finite subset of $G$ such that it is contained in a single coset of $F$ in $G$, then Lemma \ref{Lemma:CoMinimalTranslate} combined with Theorems \ref{Thm:FiniteSubsetIsAMinComp}, \ref{Thm:CoMinimal} imply that $A$ is a minimal complement to some subset of $G$ and it is a part of some co-minimal pair in $G$. 
\end{remark}

\begin{proof}[Proof of Theorem \ref{Thm:NecessaryCond}]
Let $g_1, \cdots, g_k$ be elements of $G$ and assume that $(G\setminus \{g_1, \cdots, g_k\}, S)$ is a co-minimal pair for some non-empty subset $S$ of $G$. Replacing the set $S$ by one of its right translates, we may assume that $S$ contains $e$. Since $(G\setminus \{g_1, \cdots, g_k\})\cdot e$ does not contain any one of $g_1, \cdots, g_k$, it follows that for each $1\leq i\leq k$, there exists $h_i \in G\setminus \{g_1, \cdots, g_k\}$ and $s_{i}\in S\setminus \{e\}$ such that $g_i = h_i s_{i}$. 
Let $s$ be an element of $S$ other than $e$. 
Since $G$ contains at least $3k+1$ elements, it follows that there exists an element $g\in G$ such that $g\notin \{g_1, \cdots, g_k\}\cup (\{g_1, \cdots, g_k\}\cdot s) \cup \{h_1, \cdots, h_k\}$. 
Since $g\notin \{g_1, \cdots, g_k\}\cdot s$ and $s\neq e$, it follows that $g\in (G\setminus \{g, g_1, \cdots, g_k\})\cdot s$. 
For any $1\leq i \leq k$, the elements $h_i$ lies in $G\setminus \{g, g_1, \cdots, g_k\}$ and hence $g_i \in (G\setminus \{g, g_1, \cdots, g_k\}) \cdot S$. Since $e\in S$, it follows that $G\setminus \{g, g_1, \cdots, g_k\}$ is a complement to $S$. This implies that $G\setminus \{g_1, \cdots, g_k\}$ is not a minimal complement to $S$ and consequently $(G\setminus \{g_1, \cdots, g_k\}, S)$ is not a co-minimal pair.

Let $X$ be a proper subset of $G$ containing at least $\frac 23 |G| + \frac 13$ elements. Note that 
$$
|G| - |X| 
\leq 
|G| - \frac 23 |G| - \frac 13
= \frac {|G|-1}3,
$$
which shows that $G$ contains at least $3(|G| - |X|) + 1$ elements. Hence $X$ is not a part of a co-minimal pair. 
\end{proof}

\begin{remark}
We contrast Theorem  \ref{Thm:FiniteSubsetIsAMinComp} with Theorem \ref{Thm:NecessaryCond}. In the case of free abelian groups, a finite set is always a minimal complement to some subset while this is no longer the case for finite groups. It will be interesting to characterise the groups in which given any finite subset, it is always a minimal complement to some subset. Note that by Theorem \ref{finMin}, the direct question of Nathanson \cite[Problem 12]{Nat11}, i.e., whether any finite subset admits a minimal complement or not, has an answer in the affirmative in any group. This emphasises the fact that the dual question is harder to answer even in the fairly simpler case of finite subsets.
\end{remark}

Further, in finite groups, the following question is natural. 

\begin{questionIntro}
\quad 
\begin{enumerate}
\item For each $n\geq 1$, determine the set $\calS^\cyc_n$ such that for any $k\in \calS^\cyc_n$, some subset of size $k$ of $\bbZ/n\bbZ$ is a part of a co-minimal pair. 

\item For each $n\geq 1$, determine the set $\calA^\cyc_n$ such that for any $k\in \calA^\cyc_n$, any subset of size $k$ of $\bbZ/n\bbZ$ is a part of a co-minimal pair. 

\item For each $n\geq 1$, determine the set $\calS^\ab_n$ such that for any $k\in \calS^\ab_n$, any abelian group of order $n$ contains a subset of size $k$ which is a part of a co-minimal pair. 

\item For each $n\geq 1$, determine the set $\calA^\ab_n$ such that for any $k\in \calA^\ab_n$, any subset of size $k$ of any abelian group of order $n$ is a part of a co-minimal pair. 

\item For each $n\geq 1$, determine the set $\calS_n$ such that for any $k\in \calS_n$, any group of order $n$ contains a subset of size $k$ which is a part of a co-minimal pair. 

\item For each $n\geq 1$, determine the set $\calA_n$ such that for any $k\in \calA_n$, any subset of size $k$ of any group of order $n$ is a part of a co-minimal pair. 
\end{enumerate}
\end{questionIntro}

\begin{remark}
By Theorem \ref{Thm:NecessaryCond}, the maximum of each of the above sets (excluding $n$) is $< \frac 23n + \frac 13$ (for $n\geq 2$).
\end{remark}

Further, Theorem \ref{Thm:NecessaryCond} has interesting implications in the case of infinite groups.
\begin{corollary}\label{Cor:NecessaryCond}
	Let $G$ be an infinite group. Let $A$ be a subset of $G$ such that $G\setminus A = A'$ where $A'$ is a finite set, i.e., the set-theoretic complement of $A$ is finite. Then $A$ cannot be a minimal complement to any set in $G$.
\end{corollary}

\begin{proof}
	The proof follows from the proof of Theorem \ref{Thm:NecessaryCond} noting the fact that since $G$ is infinite, for each integer $k\geq 1$, the cardinality of $G$ is $> 3k+1$. Hence if $A'$ is a finite set with $|A'| = k$, then by the hypothesis $A$ has exactly $k$ elements in its set theoretic complement. Consequently, $A$ cannot be a minimal complement to any set in $G$.  
\end{proof}

One can contrast the above result with \cite[Theorem A(2)]{BS20}, which states that any subset $A$ of an infinite group $G$ which has finite set-theoretic complement in $G$ admits a minimal complement. 

Moving on to infinite groups, Theorem \ref{Thm:CoMinimal} gives us co-minimal pairs $(A, B)$ with $A $ a non-empty finite set and $B$ an infinite set. It is natural to ask about the existence of infinite co-minimal pairs, i.e., infinite subsets $A,B$ forming a co-minimal pair. This is open in general, but certain well-behaved sets can be easily seen to form such pairs.
\begin{lemma}
Let  $A, B$ be two infinite subgroups of an abelian $G$ such that $A\times B$ is isomorphic to $G$ (under the product of inclusion maps), then $(A, B), (B, A)$ are co-minimal pairs.
\end{lemma} 
\begin{proof}
The proof is immediate.
\end{proof}

To have more examples of co-minimal pairs in $\bbZ^n$, note that if $M$ is an idempotent element of $\gln_n(\bbZ)$ other than the identity, then the subsets $$A=\ker M, B= \ker (I_n-M)$$ of $\bbZ^n$ form a co-minimal pair. 
There are examples of co-minimal pairs which are not of this form. 
For instance, if $H$ is a subgroup of a group $G$ and $\{g_\lambda\}_{\lambda\in \Lambda}$ is a set of distinct left coset representatives of $H$ in $G$, then $$(\{g_\lambda\}_{\lambda\in \Lambda}, H)$$ is a co-minimal pair. 
More involved constructions of infinite co-minimal pairs in certain abelian groups are provided in section \ref{sec4}.

\section{Some properties of co-minimal pairs}\label{sec3}

In this section, we state several important properties of co-minimal pairs. The first property is that their existence depends on the subgroup in which they are embedded, i.e, if $A\subsetneq H \subsetneq G $, with $H$ being a subgroup of a group $G$, it is possible that $A$ is a part of a co-minimal pair in $G$ but not in $H$. This is in contrast with the property of existence of minimal complements which was independent of the subgroups in which the given subset is embedded, cf. \cite[Theorem A]{MinComp2}. As an example, $(\{e^{\pi i /2}, 1, e^{-\pi i /2}\}, \{1, e^{\pi i /4}, e^{-\pi i /2}, e^{-\pi i /4}\})$ is a co-minimal pair in the group $\mu_8$ of eigth roots of unity in $\bbC$. But $\{e^{\pi i /2}, 1, e^{-\pi i /2}\}$ is not a part of co-minimal pair in the group $\mu_4$ of fourth roots of unity in $\bbC$. Thus, the obvious analogue of \cite[Theorem A]{MinComp2} does not hold in the context of co-minimal pairs. 

\begin{lemma}
Let $X$ be a nonempty subset of a group $G$. If $X$ contains a left (resp. right) translate of itself as a proper subset, i.e., $X'\subsetneq S$ where $X' = gX$ (resp. $X' = Xg$) for some $g\in G$, then $(X, Y)$ (resp. $(Y, X)$) is not co-minimal pair for any nonempty subset $Y$ of $G$. Consequently, for an element $x\in G$, the set $\{x, x^2, x^3, \cdots\}$ is a part of a co-minimal pair if and only if $x$ is of finite order. 
\end{lemma}

\begin{proof}
Suppose $X$ contains a left translate of itself as a proper subset, i.e., $X'\subsetneq S$ where $X' = gX$ for some $g\in G$. For any nonempty subset $Y$ of $G$ satisfying $X\cdot Y = G$, it follows that $X' \cdot Y = gX \cdots Y = gG = G$. Thus $X$ is not a left minimal complement to any nonempty subset $Y$ of $G$. Hence $(X, Y)$ is not a co-minimal pair for any nonempty subset $Y$ of $G$. 
Similarly, it follows that if $X$ contains a right translate of itself as a proper subset, i.e., $X'\subsetneq S$ where $X' = Xg$ for some $g\in G$, then $(Y, X)$ is not co-minimal pair for any nonempty subset $Y$ of $G$. The first statement follows. 

Assume that the order of $x$ is infinite. Then the set $\{x, x^2, x^3, \cdots\}$ contains a translate of itself as a proper subset since $x\cdot \{x, x^2, x^3, \cdots\}\subsetneq \{x, x^2, x^3, \cdots\}$. So the set $\{x, x^2, x^3, \cdots\}$ is not a part of a co-minimal pair. 

Assume that the order of $x$ is finite. Let $H$ denote the subgroup of $G$ generated by $x$. Note that the set $\{x, x^2, x^3, \cdots\}$ is equal to $H$. Let $R$ denote a subset of $G$ consisting of distinct left coset representatives of $H$ in $G$. Then $(R, H)$, i.e., $(R, \{x, x^2, x^3, \cdots\})$ is a co-minimal pair. 

This proves the second statement. 
\end{proof}
Finally, we show that co-minimal pairs are preserved under arbitrary cartesian products. 

\begin{proposition}\label{co-minimal cartesian}
Let $\mathbb{I}$ be a non-empty indexing set and $(A_{i},B_{i}), i\in \mathbb{I}$ be co-minimal pairs in groups $G_{i}, i\in \mathbb{I}$. Then $(\prod_{i\in \mathbb{I}}
A_{i}, \prod_{i\in \mathbb{I}}B_{i})$ is a co-minimal pair in $G = \prod_{i\in \mathbb{I}}G_{i}$.
\end{proposition}

\begin{proof}
To avoid confusion, just for the proof of this proposition, we denote complement pairs and co-minimal pairs both using $[\, , \, ]$ instead of the usual $(\, , \,)$. Let $A = \prod_{i\in \mathbb{I}}
A_{i}$ and $B = \prod_{i\in \mathbb{I}}
B_{i}$. For $i\in \mathbb{I},$ let $[A_{i},B_{i}]$ be co-minimal pairs in the groups $G_{i}$. Then
\begin{align*}
A_{i}\cdot B_{i} & = G_{i}, \,\, (A_{i}\setminus \lbrace a \rbrace ) \cdot B_{i} \subsetneq G_{i} \,\,\forall \,a\in A_{i} \text{ and } \,\, A_{i}\cdot (B_{i}\setminus \lbrace b\rbrace ) \subsetneq G_{i} \,\,\forall \,b\in B_{i}.
\end{align*}
Now $\prod_{i\in \mathbb{I}}
A_{i}, \prod_{i\in \mathbb{I}}
B_{i}\subseteq \prod_{i\in \mathbb{I}}
G_{i}$. It is clear that 
$$(\prod_{i\in \mathbb{I}}
A_{i})\, \cdot \,(\prod_{i\in \mathbb{I}}
B_{i}) = \prod_{i\in \mathbb{I}}
(A_{i}.B_{i}) = \prod_{i\in \mathbb{I}}
G_{i}.$$ Thus $[\prod_{i\in \mathbb{I}}
A_{i}, \prod_{i\in \mathbb{I}}
B_{i}]$ or $[A,B]$ forms a complement pair in $\prod_{i\in \mathbb{I}}
G_{i}$.

To show that it is a co-minimal pair, let us remove an element $b = \prod_{j\in \mathbb{I}}
b_{j}$ from $B$ and look at the set $ B \setminus \lbrace b\rbrace $. We show that $ B \setminus \lbrace b\rbrace $ is not a complement to $A$ in $G$, i.e., $A\cdot (B \setminus \lbrace b\rbrace)\subsetneq G$. For each $i\in \mathbb{I}$, since $B_{i}$ is a minimal complement to $A_{i}$, $\exists a_{i}\in A_{i},g_{i}\in G_{i}$ such that the only way of representing $g_{i}$ in $A_{i}\cdot B_{i}$ is $a_{i}b_{i}$. It is clear that $ \prod_{i\in \mathbb{I}}
g_{i}\notin \prod_{i\in \mathbb{I}}
A_{i}\cdot (B \setminus \lbrace b\rbrace)$ because $  \prod_{i\in \mathbb{I}}
g_{i} $ can only be represented in $(\prod_{i\in \mathbb{I}}
A_{i})\cdot (\prod_{i\in \mathbb{I}}
B_{i})$ as $  \prod_{i\in \mathbb{I}}
a_{i}b_{i}$. Thus $ \prod_{i\in \mathbb{I}}
B_{i}$ is a minimal complement to $ \prod_{i\in \mathbb{I}}
A_{i}$. An exactly similar argument shows that $ \prod_{i\in \mathbb{I}}
A_{i}$ is also a minimal complement to $ \prod_{i\in \mathbb{I}}
B_{i}$. This shows that $[\prod_{i\in \mathbb{I}}
A_{i}, \prod_{i\in \mathbb{I}}
B_{i}]$ or $[A,B]$ forms a co-minimal pair in $\prod_{i\in \mathbb{I}}
G_{i}$. 
\end{proof}

\section{Spiked subsets and co-minimality}\label{sec4}

The above construction of co-minimal pairs was in the context of cartesian product of groups. If instead we take product groups, then the following can be established. Suppose $G_1, G_2$ are subgroups of an abelian group $G$ such that the map 
$$G_1\times G_2 \rightarrow G \quad \text{ defined by } (g_1, g_2)\mapsto g_1g_2$$ 
is an isomorphism. We identify the group $G$ with $G_1\times G_2$ via this isomorphism. The subsets of $G$ of the form $B\times B'$, more specifically, the subsets of $G$ of the form $B\times G_2$ are one of the simplest subsets of $G$. 

\begin{lemma}\label{Lemma:Spike}
The subsets of $G$ of the form $B \times G_{2}$ with $B\subseteq G_1$ is a part of a co-minimal pair in $G$ if and only if $B$ is a part of co-minimal pair in $G_1$. 
\end{lemma}

\begin{proof}
Suppose $B\times G_2$ is a part of a co-minimal pair $(B\times G_2, S)$ in $G$. Denote the projection map $G \to G_1$ by $\pi$. Note that the image $\pi(S)$ is a minimal complement to $B$. Since $B\times G_2$ is a minimal complement to $S$, it follows that $B$ is a minimal complement to $\pi(S)$. So $(B, \pi(S))$ is a co-minimal pair. 

Suppose $(B, M)$ is a co-minimal pair for some subset $M$ of $G_1$. Then $(B\times G_2, M\times \{0\})$ is a co-minimal pair. 
\end{proof}

The subsets of $G$ of the form $B\times G_2$ are examples of a more general class of subsets, called `spiked subsets' as introduced in \cite{MinComp2}. For the sake of completeness, we recall its definition. 
A subset $X$ of $\bbZ^{k+1}$ is called a \textit{spiked subset} if  
$$\calB\times \bbZ
\subseteq
X
\subseteq 
(\calB\times \bbZ )
\bigsqcup 
\left(
\sqcup 
_{x\in \bbZ^k\setminus \calB}
\left(
\{
x
\}
\times 
(-\infty, u(x))
\right)
\right)$$
holds for some nonempty subset $\calB$ of $\bbZ^k$ and some function $u:\bbZ^k\to \bbZ$. The set $\calB$ is called the \textit{base} of $X$. 
We will say that such a set $X$ is a $u$-bounded spiked subset with base $\calB$. 
By \cite[Lemma 4.5]{MinComp2}, any function $u:\bbZ^k\to \bbZ$ admits a moderation $v$, i.e., a function $v: \bbZ^k\to \bbZ$ such that for each $x_0\in \bbZ^k$, the function 
$$x\mapsto u(x) + v(x_0-x)$$
defined on $\bbZ^k$ is bounded above.
It turns out that any (or some) $u$-bounded spiked subset with base $\calB$ admits a minimal complement in $\bbZ^{k+1}$ if and only if the base $\calB$ admits a minimal complement in $\bbZ^k$ (see \cite[Theorems 4.6, 5.6]{MinComp2}). 

More generally, for an abelian group $\calG$ as above with subgroups $G_1, G_2$ such that $G_2$ is free and the multiplication map from $$G_1\times G_2\to \calG$$ is injective, the notion of `spiked subsets' can be extended (cf. \cite[Definitions 5.1, 5.2]{MinComp2}). 

\begin{definition}
\label{Spiked subsets}
A subset $X$ of an abelian group $\calG$ is said to be a $(u, \varphi)$-\textnormal{bounded spiked subset} with respect to subgroups $G_1, G_2$ of $\calG$ if 
\begin{enumerate}
\item $G_2$ is a finitely generated free abelian group of positive rank,
\item the homomorphism $G_1\times G_2 \to G$ defined by $(g_1, g_2) \mapsto g_1g_2$ is injective,
\end{enumerate}
and there exists a function $u:G_1\to G_2$ and an isomorphism $\varphi:G_2\xra{\sim} \bbZ^{\mathrm{rk} G_2}$ such that 
$$\calB  G_2
\subseteq
X
\subseteq 
\calB G_2
\bigsqcup 
\left(
\bigsqcup_{g_1\in G_1\setminus \calB}
g_1\cdot
\left(
\varphi^\mo \left(\bbZ^{\mathrm{rk} G_2}_{<\varphi(u(g_1))}
\right)
\right)
\right)
$$
holds for some non-empty subset $\calB$ of $G_1$. The set $\calB$ is called the \textnormal{base} of $X$. 
\end{definition}
The notion of moderation extends to such a context (cf. \cite[Definition 5.3]{MinComp2}). Moreover, when $G_1$ is finitely generated, it follows that $u$ admits a $\varphi$-moderation $v$ (cf. \cite[Proposition 5.4]{MinComp2}). 
Furthermore, if $G_1$ is finitely generated, then a $(u, \varphi)$-bounded spiked subset of $\calG$ with respect to $G_1, G_2$ having base $\calB$ admits the graph of the restriction of some $\varphi$-moderation of $u$ to some subset of $G_1$ as a minimal complement in $G_1G_2$ if and only if the base $\calB$ admits a minimal complement in $G_1$ (see \cite[Theorem 5.6]{MinComp2}). 

The following result states that an appropriate formulation of Lemma \ref{Lemma:Spike}  also holds for spiked subsets, and thereby classifies all the spiked subsets which can be a part of a co-minimal pair of certain form. 

\begin{theorem}
\label{Thm:Spiked}
Let $G_1, G_2$ denote two subgroups of an abelian group $\calG$. Let $X$ be a $(u, \varphi)$-bounded spiked subset of $\calG$ with respect to $G_1, G_2$ and with base $\calB$. If $u$ admits a $\varphi$-moderation, then $X$ is a part of a co-minimal pair in $G_1G_2$ of the form $(X, M_v)$ where $M_v$ is the graph of the restriction of a moderation $v$ of $u$ to some subset $M$ of $G_1$ if and only if $X$ is equal to $\calB G_2$ and $\calB$ is a part of a co-minimal pair in $G_1$. 
\end{theorem}

\begin{proof}
Assume that there exists a $\varphi$-moderation $v$ of $u$ such that $(X, M_v)$ is co-minimal pair in $G_1G_2$ where $M_v$ denotes the graph of the restriction of $v$ of $u$ to some subset $M$ of $G_1$. Since $M_v$ is a minimal complement to $X$, it follows that $M$ is a minimal complement to $\calB$. By \cite[Theorem 5.6]{MinComp2}, $M_v$ is a minimal complement to $\calB G_2$ in $G_1G_2$. Since $(X, M_v)$ is a co-minimal pair in $G_1G_2$, we conclude that $X$ cannot be larger than $\calB G_2$. Hence $X$ is equal to $\calB G_2$. If $\calB$ is not a minimal complement to $M$, then $\calB G_2$ is not a minimal complement to $M_v$. Since $(X, M_v)$ is a co-minimal pair in $G_1G_2$, it follows that $\calB$ is a minimal complement to $M$, i.e., $\calB$ is a part of a co-minimal pair in $G_1$. 

Suppose $\calB$ is a part of a co-minimal pair $(\calB, M)$ in $G_1$. Let $v$ denote a $\varphi$-moderation of $u$. Then the graph $M_v$ of the restriction of $v$ to $M$ is a minimal complement to $\calB G_2$ by \cite[Theorem 5.6]{MinComp2}. If $\calB G_2$ were not a minimal complement to $M_v$, then the set $$(\calB G_2)\setminus\{b+t\}$$ would be a complement to $M$ for some elements $b\in \calB$ and $t \in G_2$. Since $\calB$ is a minimal complement to $M$, it follows that $M$ contains an element $m$ such that $b+m$ does not belong to $(\calB\setminus \{b\})+M$. This implies that $$((\calB G_2)\setminus\{b+t\})+M_v$$ does not contain the element $$b+m+ t+v(m)$$ of $G_1G_2$. Consequently, $\calB G_2$ is a minimal complement to $M_v$.
\end{proof}

\section{Semilinear sets, approximate subgroups and non co-minimality}\label{sec5}
We conclude the discussion on co-minimal pairs by mentioning subsets which are in a sense the other extreme of being part of a co-minimal pair, i.e., they \emph{do not belong to any minimal pair}. In other words, they are not a minimal complement to any subset and also no subset can be a minimal complement to one of these sets. For this we recall the well-defined notion of an arithmetic progression in an abelian group.

\begin{definition}[Arithmetic progressions]\label{ArithmeticProgression}
A subset $X$ of an abelian group $(G,+)$ is an \textnormal{unbounded arithmetic progression in $G$} if there exist $a \in G$ and $b\in G\setminus \{e\}$ such that
$$X = P(a,b) := \{ a + n b \, | \,n \in \bbZ_{\geq 0} \}.$$
A subset $Y$ of $G$ is a \emph{bounded arithmetic progression} if there exist $a,b \in G$ and $m \in \bbZ_{\geq 0}$ such that 
$$Y = P_m(a,b) := \{ a + n b \,|\, n \in [0,m] \cap \bbZ \}.$$
\end{definition}

More generally, we will use the following objects.

\begin{definition}[Generalised arithmetic progressions]
An infinite subset $\mathbf{X}$ of an abelian group $(G,+)$ is an \emph{unbounded generalised arithmetic progression of dimension $d$ with respect to $b_1, \cdots, b_d\in G$} if there exists an element $a \in G$ such that $\mathbf X$ is equal to 
$$\{ a + n_{1}b_{1}+\cdots + n_{d}b_{d} \, | \,n_i \text{ runs over } F_i \}$$
for some subsets $F_1, \cdots, F_d$ of $\bbZ$ where each $F_i$ is either an unbounded arithmetic progression in $\bbZ$ or a finite set. 
A subset $\mathbf{Y}$ of $G$ is a \emph{bounded generalised arithmetic progression of dimension $d$} if there exist $a,b_{1},\cdots, b_{d} \in G$ and $m \in \bbZ_{\geq 0}$ such that 
$$\mathbf{Y} =  P_{m_{1},\cdots, m_{d}}(a,b_{1},\cdots, b_{d}) := \{ a + n_{1}b_{1}+\cdots + n_{d}b_{d} | \, n_{i} \in [0,m_{i}] \cap \bbZ, 1\leqslant i\leqslant d \}.$$	
\end{definition}

A generalised arithmetic progression is also known as a linear set. A finite union of unbounded linear sets is called a semilinear set.

\begin{theorem}\label{semilinear}
Let $G$ be a free abelian group and $A$ be a non-empty subset of $G$. The following are true
\begin{enumerate}
\item If $A = P_m(a,b)$ or in general if $A= P_{m_{1},\cdots, m_{d}}(a,b_{1},\cdots, b_{d}) $, then $A$ is always a minimal complement to some subset of $G$ and also $A$ has a minimal complement.
\item If $A =  P(a,b)$, or more generally, if $A$ is an unbounded generalised arithmetic progression of dimension $d$ with respect to $b_1, \cdots, b_d\in G$ and $b_1, \cdots, b_d$ generate a free subgroup of $G$ of rank $d$, then $A$ is neither a minimal complement to any subset of $G$, nor does it have a minimal complement in $G$.
\item If $G=\mathbb{Z}^{\mathbb{I}}$, $A=  P(a,b)^{\mathbb{I}}$ where $\mathbb{I}$ is some indexing set and $ P(a,b) \subseteq \mathbb{Z}$, then $A$ is neither a minimal complement to any subset of $G$, nor does it have a minimal complement in $G$.
\end{enumerate}
Thus the subsets in (2) and (3) above, can never belong to a minimal pair.
\end{theorem}

\begin{proof}
(1) Note that $A$ is a non-empty finite subset in $G$. It follows from Theorem \ref{Thm:FiniteSubsetIsAMinComp} that $A$ is a minimal complement to some subset and from Theorem \ref{finMin} that $A$ admits a minimal complement.

(2) Suppose $A$ is an unbounded generalised arithmetic progression of dimension $d$ with respect to $b_1, \cdots, b_d\in G$ and $b_1, \cdots, b_d$ generate a subgroup of $G$ of rank $d$. Then 
$$A = \{ a + n_{1}b_{1}+\cdots + n_{d}b_{d} \, | \,n_i \text{ runs over } F_i \}$$
for some element $a\in G$ and some subsets $F_1, \cdots, F_d$ of $\bbZ$ where each $F_i$ is either an unbounded arithmetic progression in $\bbZ$ or a finite set. Reordering the elements $b_1, \cdots, b_d$ (if necessary), we assume that $F_1, \cdots, F_e$ are unbounded arithmetic progressions in $\bbZ$ with $1\leq e\leq d$ and the remaining $F_i$'s are finite. Replacing $a$ by a suitable element of $G$, we can assume that the initial term of each of $F_1, \cdots, F_e$ is equal to zero. Let $b_1', \cdots, b_e'$ denote nonzero elements of $\bbZ$ such that $F_i$ is equal to $\bbZ_{\geq 0} b_i'$ for $1\leq i \leq e$. Note that $A + b_1'b_1$ is a proper subset of $A$ and hence $A$ is not a minimal complement to any subset of $G$. 

It remains to show that $A$ does not have a minimal complement in $G$. Replacing $A$ by a translate of $A$ (if necessary), we can assume $a=0$. By \cite[Theorem 2.3]{MinComp2}, it is enough to show that $A$ does not admit a minimal complement in the subgroup 
$$\langle b_1'b_1, \cdots, b_e'b_e, b_{e+1}, \cdots, b_d \rangle .$$ 
Hence, the elements, $$b_1'b_1, \cdots, b_e'b_e, b_{e+1}, \cdots, b_d$$ could be identified with $e_1, \cdots, e_d$ and $A$ could be thought of as a subset of $\bbZ^d$. If $e$ is equal to $d$, then $A$ is equal to the subset $\bbZ_{\geq 0} ^d$ of $\bbZ^d$. By \cite[Corollary 3.2(2)]{MinComp1}, the set $A$ does not admit a minimal complement in $\bbZ^d$. Suppose $e$ is less than $d$. Let $$\pi_1: \bbZ^d\to \bbZ^e \text{ and } \pi_2: \bbZ^d \to \bbZ^{d-e} $$ denote the projections onto the first $e$ coordinates and onto the last $d-e$ coordinates respectively. Suppose $F_{e+1}, \cdots, F_d$ are contained in the intervals $[p_{e+1}, q_{e+1}], \cdots, [p_d, q_d]$ respectively. Let us assume that $B\subseteq \bbZ^d$ is a minimal complement of $A$. For each $$v\in [-q_{e+1}, -p_{e+1}]\times \cdots \times [-q_d, -p_d]$$ such that the set $B\cap \pi_2^\mo(v)$ is non-empty, choose an element 
$$(x_{v1}, \cdots, x_{ve}) \in \pi_1(B \cap \pi_2^\mo (v)).$$
Since $B$ is a minimal complement to $A$, it follows that the set $\pi_1(B \cap \pi_2^\mo (v))$ does not contain any point whose $i$-coordinate is less than or equal to $x_{vi}-1$ for each $1\leq i \leq e$. Since $(\underbrace{-1, \cdots, -1}_{e\text{-times}}, \underbrace{0, \cdots, 0}_{(d-e)\text{-times}})$ belongs to $\bbZ^d = B+A$, it follows that the set $B\cap \pi_2^\mo(v)$ is non-empty for some  $$v\in [-q_{e+1}, -p_{e+1}]\times \cdots \times [-q_d, -p_d].$$ For each $1\leq i\leq e$, define 
\begin{equation*}
x_i : = 
\min_{ v\in [-q_{e+1}, -p_{e+1}]\times \cdots \times [-q_d, -p_d], B \cap \pi_2^\mo(v) \neq \emptyset} (x_{vi}-1).
\end{equation*}
Note that $(x_1, \cdots, x_e, \underbrace{0, \cdots, 0}_{(d-e)\text{-times}})$ does not belong to $B+A=\bbZ^d$. Hence $A$ does not admit any minimal complement. 

(3) Let $f\in \bbZ^{\bbI}$ denote the constant function which takes the value $b$. Then $A+f$ is a proper subset of $A$ and hence $A$ cannot be a minimal complement to some subset of $\bbZ^\bbI$. Let us assume that $b$ is positive and $a$ is equal to $0$. Suppose $A$ admits a minimal complement $B$ in $\bbZ^\bbI$. Let $\{c_i\}_{i\in \bbI}$ denote an element of $B$. Since $B$ is a complement to $A$ and $B+A$ contains $\{c_i-1\}_{i\in \bbI}$, it follows that $B$ contains an element $\{d_i\}_{i\in \bbI}$ with $d_i \leq c_i-1$ for all $i\in \bbI$. Note that $B \setminus \{\{c_i\}_{i\in \bbI}\}$ is also a complement to $A$. Hence $A$ does not admit any minimal complement in $\bbZ^\bbI$. 
\end{proof}

As a corollary, we deduce the following fact.
\begin{corollary}\label{corSemilinear}
There exist semilinear sets $A$  such that $A$ is neither a minimal complement to any subset of $G$, nor does it have a minimal complement in $G$.
\end{corollary}

\begin{proof}
Take $A$ to be one of the sets described in $(2)$ of Theorem \ref{semilinear}.
\end{proof}

The above Theorem \ref{semilinear} and Corollary \ref{corSemilinear} shed light on subtle differences in existence and inexistence of minimal complements in general abelian groups. We have seen that a proper subgroup $H$ in any group $G$ is always a minimal complement to any of its coset class and it also admits a minimal complement. However, the fact does not remain necessarily true when we pass to subsets which are close to being subgroups. Let us recall the notion of an approximate subgroup. 

\begin{definition}[$K$-approximate subgroup]\label{Appgr}
Let $G$ be a group and $K\geqslant 1$ be some parameter. A finite set $A\subseteq G $ is called a \textnormal{$K$-approximate group} if 
\begin{enumerate}
\item identity of $G$, $e\in A$,
\item it is symmetric, i.e., if $a\in A$ then $a^{-1}\in A$,
\item there is a symmetric subset $X$ lying in $A\cdot A$ with $|X|\leqslant K$ such that $A\cdot A\subseteq X\cdot A$.
\end{enumerate}
\end{definition} 

The formal definition of an approximate subgroup was introduced by Tao in \cite{Tao08}. Informally, these sort of subsets have been studied since the time of Fre{\u\i}man \cite{Fre64}. Nathanson considered a more general notion of an approximate group. For him, the set $A$ need not be finite, nor symmetric, nor contain the identity. 
\begin{definition}[$(r,l)$-approximate group \cite{Nat18}]\label{NatAppgr}
Let $r,l \in \mathbb{N}$ with $r\geqslant 2$. A non-empty subset $A\subseteq G$ is an $(r,l)$-approximate group if there exists a set $X\subseteq G$ such that $$|X|\leqslant l \text{ and } A^{r}\subseteq X\cdot A.$$
\end{definition}

Any finite approximate subgroup always belongs to a minimal pair (by {\cite[Theorem 2.1]{MinComp1}}), while inside free abelian groups they always belong to some co-minimal pair by Theorem \ref{Thm:CoMinimal}. When we pass to infinite approximate subgroups (in the sense of Nathanson), this is not necessarily the case. By $(2)$ of Theorem \ref{semilinear}, there exist unbounded linear sets which can never belong to a minimal pair while Corollary \ref{corSemilinear} concludes the same about semilinear sets. Unbounded linear sets are however examples of approximate subgroups in the sense of Definition \ref{NatAppgr}.

We mention briefly that in the same paper \cite{Nat18}, Nathanson introduced the notion of an asymptotic approximate group, which is a subset $A\subseteq G$ such that every sufficiently large power of $A$ is an $(r,l)$-approximate group.
\begin{definition}[Asymptotic $(r,l)$-approximate group]\label{NatAAppgr}
Let $r,l \in \bbN$, with $r\geqslant 2$. A subset $A$ of a group $(G,\cdot)$ is an asymptotic $(r,l)$-approximate group if there exists a threshold $h_0 \in \bbN$ such that for each natural number $h \geq h_0$, there exists a subset $X_h$ of $G$ satisfying 
$$|X_{h}|\leqslant l \text{ and } A^{hr}\subseteq X_{h}\cdot A^{h}.$$
\end{definition}

\begin{enumerate}
\item Nathanson in \cite{Nat18} showed that any finite subset in an abelian group is an asymptotic $(r,l)$-approximate group for some $r,l\in \bbN$. 
\item In \cite{BM19}, it was shown that unbounded linear sets and also semilinear sets are asymptotic $(r,l)$-approximate groups for some $r,l\in \bbN$.
\end{enumerate}
In the first case, the sets belong to co-minimal pairs while in the second case, as a consequence of Corollary \ref{corSemilinear}, they do not necessarily belong to even minimal pairs, let alone co-minimal pairs.
Thus, in general, it is also not true that asymptotic approximate groups will be part of some co-minimal pair. 

\section{Concluding remarks and further questions}\label{sec6}

The above discussion motivates one to consider co-minimal pairs with $A=B$, i.e., non-empty subsets $A$ of an abelian group $G$ with $A+A = G$ and $A+A\setminus\lbrace a\rbrace \subsetneq G, \forall \,a\in A$. These type of sets have also been considered by Kwon in the context of $G = \mathbb{Z}$ (minimal self-complements). He showed that a set $A$ has such properties if and only if $A$ avoids $3$ term arithmetic progressions and $A + A = \bbZ$, and he gave a construction of one such set in $\mathbb{Z}$. We remark that this holds in the context of an arbitrary abelian group (with the natural notion of an arithmetic progression in these groups, see Definition \ref{ArithmeticProgression}). The proof is a natural extension of the proof of Kwon for $G=\mathbb{Z}$ (see \cite[Theorem 10]{Kwon19}) and hence omitted.

\begin{proposition}\label{prop:Kwon}
Let $G$ be an abelian group. Then for a subset $A$ of $G$, $(A, A)$ is a co-minimal pair if and only if $A+A=G$ and $A$ avoids non-trivial $3$-term arithmetic progressions.
\end{proposition}

However, the construction of such sets in groups other than $\mathbb{Z}$ is an issue. Further, even in $\mathbb{Z}$, one can ask about the existence of infinite co-minimal pairs, i.e., infinite sets $A,B$ with $A\neq B$, such that $(A,B)$ is a co-minimal pair. In the following, we state a list of the open questions concerning co-minimal pairs for further research.

\begin{questionIntro}\label{question1}
Do there exist infinite subsets $A,B$ of $\mathbb{Z}$ such that $(A,B)$ is a co-minimal pair and $B$ is not a translate of $A$? For instance, what about infinite co-minimal pairs with at least one of $A$ or $B$ bounded on one side? What about the case when $A$ or $B$ contains $3$-term arithmetic progressions? 
\end{questionIntro} 

\begin{questionIntro}
	What is the behaviour of minimal complement with respect to sets which are close to being a subgroup (but is not actually a subgroup)? For instance, does there exist infinite approximate subgroup or infinite asymptotic $(r,l)$ approximate subgroup in a group $G$, which can be a minimal complement to some subset? Does the answer depend on specific values of $r,l$?
\end{questionIntro}

The following question was posed by Laurent Bartholdi during a talk of the first author at IHP, Paris. 

\begin{questionIntro}\label{question3}
Given any infinite, symmetric subset in a group $G$ or even in $\mathbb{Z}^{d}$ ($d\geqslant 2$), what can we say about the existence of its minimal complement? 
\end{questionIntro}

We recall that for $d=1$, i.e., $G = \mathbb{Z}$, the existence of a minimal complement of any infinite, symmetric subset is guaranteed by a result of Chen--Yang \cite[Theorem 1]{CY12}. Questions \ref{question1} and \ref{question3} can be combined to ask:

\begin{questionIntro}
Do there exist infinite subsets $A,B\subseteq \mathbb{Z}$, which form a co-minimal pair and at least one of them is symmetric? 
\end{questionIntro}

\section{Acknowledgements}
The first author would like to acknowledge the support of the OWLF program of the MFO, Oberwolfach and would also like to thank the Fakult\"at f\"ur Mathematik, Universit\"at Wien.
The second author would like to acknowledge the Initiation Grant from the Indian Institute of Science Education and Research Bhopal, and the INSPIRE Faculty Award from the Department of Science and Technology, Government of India.


	\providecommand{\bysame}{\leavevmode\hbox to3em{\hrulefill}\thinspace}
	\providecommand{\MR}{\relax\ifhmode\unskip\space\fi MR }
	\providecommand{\MRhref}[2]{%
	\href{http://www.ams.org/mathscinet-getitem?mr=#1}{#2}
	}
	\providecommand{\href}[2]{#2}

\end{document}